\documentclass[a4paper,10pt]{amsart}
\usepackage[UKenglish]{babel}
\usepackage{amsmath,amssymb,amsthm}
\usepackage[pdftex]{color}
\usepackage[bookmarks=true,hyperindex,pdftex,colorlinks, citecolor=MidnightBlue,linkcolor=MidnightBlue, urlcolor=MidnightBlue
]{hyperref}
\usepackage[dvipsnames]{xcolor}

\usepackage{enumerate}

\theoremstyle{plain}
\newtheorem{teo}{Theorem}[section]
\newtheorem{prop}[teo]{Proposition}
\newtheorem{lem}[teo]{Lemma}

\newtheorem{cor}[teo]{Corollary}

\theoremstyle{definition}
\newtheorem{defi}[teo]{Definition}

\newtheorem{ejem}[teo]{Example}

\theoremstyle{remark}
\newtheorem{remark}[teo]{Remark}

\renewcommand{\geq}{\geqslant}
\renewcommand{\leq}{\leqslant}

\newcommand{\N}{\mathbb{N}}
\newcommand{\R}{\mathbb{R}}

\newcommand{\SA}{\operatorname{LipSNA}}

\newcommand{\Lip}{{\mathrm{Lip}}_0}
\newcommand{\Lipc}{{\mathrm{Lip}}_{0K}}
\newcommand{\diam}{\operatorname{diam}}

\newcommand{\cco}{\overline{\operatorname{co}}}

\newcommand{\Mol}{\operatorname{Mol}}
\newcommand{\F}{\mathcal{F}}

\newcommand{\eps}{\varepsilon}

\begin{document}

\title[Stability of the Bishop--Phelps--Bollob\'{a}s property for Lipschitz maps]{Some stability properties for the Bishop--Phelps--Bollob\'{a}s property for Lipschitz maps}

\author[R.~Chiclana]{Rafael Chiclana}

\author[M.~Mart\'in]{Miguel Mart\'in}

\address{Universidad de Granada, Facultad de Ciencias.
Departamento de An\'{a}lisis Matem\'{a}tico, 18071-Granada
(Spain)}
\email{rchiclana@ugr.es, mmartins@ugr.es}

\thanks{Research partially supported by projects PGC2018-093794-B-I00 (MCIU/AEI/FEDER, UE) and FQM-185 (Junta de Andaluc\'{\i}a/FEDER, UE)}



\date{April 20th, 2020}

\begin{abstract}
	We study the stability behavior of the Bishop-Phelps-Bollob\'{a}s property for Lipschitz maps (Lip-BPB property). This property is a Lipschitz version of the classical Bishop-Phelps-Bollob\'{a}s property and deals with the possibility of approximating a Lipschitz map that almost attains its (Lipschitz) norm at a pair of distinct points by a Lipschitz map attaining its norm at a pair of distinct points (relatively) very closed to the previous one. We first study the stability of this property under the (metric) sum of the domain spaces. Next, we study when it is possible to pass the Lip-BPB property from scalar functions to some vector-valued maps, getting some positive results related to the notions of $\Gamma$-flat operators and $ACK$ structure.
	We get sharper results for the case of Lipschitz compact maps. The behaviour of the property with respect to absolute sums of the target space is also studied. We also get results similar to the above ones about the density of strongly norm attaining Lipschitz maps and of Lipschitz compact maps.
\end{abstract}

\maketitle

\thispagestyle{plain}

\section{Introduction}
A \textit{pointed metric space} is just a metric space $M$ in which we have distinguished an element, denoted by $0$, that acts like the center of the metric space. All along this paper, metric spaces will be complete and Banach spaces will be over the real scalars. Given a Banach space $X$, the closed unit ball and the sphere of $X$ will be denoted by $B_X$ and $S_X$ respectively. If $Y$ is another Banach space, $\operatorname{L}(X,Y)$ will denote the space of bounded linear operators from $X$ to $Y$. In the case $Y=\mathbb{R}$, we simply write $X^*$. Given a pointed metric space $M$ and a Banach space $Y$, $\Lip(M,Y)$ denotes the vector space of Lipschitz maps from $M$ to $Y$ that vanish at $0$. This vector space becomes a Banach space when it is endowed with the norm
\[  \|F\|_L = \sup \left \{\frac{\|F(p)-F(q)\|}{d(p,q)} \colon p, q \in M,\, p\neq q \right \} \quad \forall \, F \in \Lip(M,Y).\]
We say that a Lipschitz map $F\colon M \longrightarrow Y$ \textit{strongly attains its norm} if the above supremum is a maximum, that is, there exist distinct points $p$, $q \in M$ such that
\[ \frac{\|F(p)-F(q)\|}{d(p,q)}=\|F\|_L. \]
We write $\SA(M,Y)$ to denote the subset of those Lipschitz maps from $M$ to $Y$ which strongly attain their norm. If $M$ is a finite metric space, it is clear that every Lipschitz map strongly attains its norm. One the other hand, it is easy to show that for any infinite metric space $M$ it is possible to find a Lipschitz function $F\colon M \longrightarrow \mathbb{R}$ which does not strongly attain its norm (see Corollary 3.46 in \cite{wea2}).

In view of this, a natural question appears: Is it always possible to approximate a Lipschitz map by strongly norm attaining Lipschitz maps? The negative answer to this question was given in \cite[Example~2.1]{kms}, where it is shown that this is not the case, among other examples, for $M=[0,1]$ and $Y=\R$. On the other hand, the first positive examples appeared in \cite[\S 5]{Godefroy-survey-2015}, which include the case when $M$ is a compact H\"{o}lder metric space and $Y$ is finite dimensional. The question is then reformulated in \cite[Problem~6.7]{Godefroy-survey-2015}: for which metric spaces $M$ and Banach spaces $Y$ is the subset $\SA(M,Y)$ dense in $\Lip(M,Y)$?  After that, several papers studying this problem have appeared (see \cite[Section 7]{lppr}, \cite{articulo1}, \cite{SNA_2}). Analyzing the results given in those papers (which we will briefly comment later on), it is possible to extract a common idea in them: a Lipschitz map is identified with a bounded linear operator between Banach spaces. This allows to apply some results of the classical theory of norm attaining linear operators to obtain results in the Lipschitz context. In order to understand this identification, we need to introduce the Lipschitz-free space over a metric space.

Let $M$ be a pointed metric space. We denote by $\delta$ the canonical isometric embedding of $M$ into $\Lip(M,\mathbb{R})^*$, which is given by
\[ \langle f, \delta(p)\rangle=f(p) \quad \forall \, p \in M, \quad \forall \, f \in \Lip(M,\mathbb{R}).\]
We denote by $\mathcal{F}(M)$ the norm-closed linear span of $\delta(M)$, that is,
\[ \mathcal{F}(M)=\overline{\operatorname{span}}\{ \delta(p) \colon p \in M \}\subseteq \Lip(M,\mathbb{R})^*,\]
which is usually called the \textit{Lipschitz-free space over} $M$. It is well known that $\mathcal{F}(M)$ is an isometric predual of the Banach space $\Lip(M,\mathbb{R})$. Indeed, in \cite{wea2} it was shown that it is the unique isometric predual when $M$ is bounded or a geodesic space. We refer for background to the papers \cite{Godefroy-survey-2015} and \cite{Gode2}, and to the book \cite{wea2} (where these spaces are called \textit{Arens-Eell spaces}). Given a Lipschitz map $F\colon M \longrightarrow Y$, we can consider the unique bounded linear operator $\widehat{F} \colon \mathcal{F}(M)\longrightarrow Y$ satisfying
\[ \widehat{F}(\delta(p))=F(p) \quad \forall \, p \in M. \]
It turns out that the mapping $F \longmapsto \widehat{F}$ is an isometric isomorphism between the spaces $\Lip(M,Y)$ and $\operatorname{L}(\mathcal{F}(M),Y)$. Therefore, we can identify every Lipschitz map from $M$ to $Y$ with a bounded linear operator from $\mathcal{F}(M)$ to $Y$. In order to reformulate our question in terms of this identification, we need to introduce the notion of \textit{molecule} of $\mathcal{F}(M)$, which is just an element of the form
\[ m_{p,q}=\frac{\delta(p)-\delta(q)}{d(p,q)},  \quad \mbox{where } p,q \in M, \, p\neq q.\]
We write $\Mol(M)$ to denote the set of all molecules of $\mathcal{F}(M)$. As a consequence of Hahn-Banach theorem, it is easy to see that every molecule has norm one and that we can recover the unit ball of $\mathcal{F}(M)$ as the closed convex hull of the molecules, that is,
\[ B_{\mathcal{F}(M)} = \cco(\Mol(M)).\]
Now, we can reformulate our question using the notation we have introduced. It is enough to notice that $F$ strongly attains its norm at a pair of distinct points $(p,q)$ of $M$ if, and only if, $\widehat{F}$ attains its norm (in the classical sense) at the molecule $m_{p,q}$, that is, $\|\widehat{F}\|=\|\widehat{F}(m_{p,q})\|$. Then, a Lipschitz map strongly attains its norm if and only if its bounded linear operator associated attains its norm at some molecule. Hence, we are studying for which metric spaces $M$ and Banach spaces $Y$ the set of those bounded linear operators from $\mathcal{F}(M)$ to $Y$ which attain their norm at some molecule is dense in $\operatorname{L}(\mathcal{F}(M),Y)$.
Furthermore, we are also interested in studying the following stronger version of such density, which was introduced in \cite{ChiMar-NA}.

\begin{defi}[\textrm{\cite{ChiMar-NA}}]\label{Lip-BPBp}
	Let $M$ be a pointed metric space and let $Y$ be a Banach space. We say that the pair $(M,Y)$ has the \emph{Lipschitz Bishop-Phelps-Bollob\'{a}s property} (\emph{Lip-BPB property} for short), if given $\varepsilon>0$ there is $\eta(\varepsilon)>0$ such that for every norm-one $F\in \Lip(M,Y)$ and every $p,q\in M$, $p\neq q$ such that $\|F(p)-F(q)\|> \bigl(1-\eta(\varepsilon)\bigr)d(p,q)$, there exist $G \in \Lip(M,Y)$ and $r,s\in M$, $r\neq s$, such that
	\[
	\frac{\|G(r)-G(s)\|}{d(r,s)}=\|G\|_L=1,\quad \|G-F\|_L<\varepsilon, \quad \frac{d(p,r)+d(q,s)}{d(p,q)}<\varepsilon.
	\]
	Equivalently (see \cite[Remark~1.2.a]{ChiMar-NA}), the pair $(M,Y)$ has the Lip-BPB property if and only if given $\varepsilon>0$ there is $\eta(\varepsilon)>0$ such that for every norm-one $\widehat F\in \operatorname{L}(\mathcal{F}(M),Y)$ and every $m\in \Mol(M)$ such that $\|\widehat F(m)\|>1-\eta(\varepsilon)$, there exist $\widehat G \in \operatorname{L}(\mathcal{F}(M),Y)$ and $u \in \Mol(M)$ such that
	\[
	\|\widehat{G}(u)\|=\|G\|_L=1,  \quad \|\widehat F-\widehat G\| <\varepsilon, \quad \|m-u\|<\varepsilon. \]
\end{defi}
If the previous definition holds for a class of linear operators from $\mathcal{F}(M)$ to $Y$, we will say that the pair $(M,Y)$ has the Lip-BPB property for that class.
This definition can be understood as a non-linear generalization of the classical Bishop-Phelps-Bollob\'{a}s property (BPBp for short). It is clear that if a pair $(M,Y)$ has the Lip-BPB property, then $\SA(M,Y)$ will be norm-dense in $\Lip(M,Y)$. On the other hand, the reciprocal result is far from being true. In fact, if $M$ is a finite pointed metric space, while it is clear that $\SA(M,Y)=\Lip(M,Y)$ for every Banach space $Y$, Example 2.5 in \cite{ChiMar-NA} shows that one can find finite pointed metric spaces $M$ and Banach spaces $Y$ such that the pairs $(M,Y)$ fail to have the Lip-BPB property. For this reason, throughout this paper each of these notions of approximation by strongly norm attaining Lipschitz maps will be discussed separately. Moreover, we will also do the same study for Lipschitz compact maps. Let $F\colon M \longrightarrow Y$ be Lipschitz. We say that $F$ is \textit{Lipschitz compact} when its Lipschitz image, that is, the set
\[ \left \{\frac{F(p)-F(q)}{d(p,q)} \colon p,q \in M, \, p\neq q \right \}\subseteq Y,\]
is relatively compact. We denote by $\Lipc(M,Y)$ the space of Lipschitz compact maps from $M$ to $Y$. Some results related to this notion appear in \cite{lipschitzcompact}. One can easily check that $F\colon M \longrightarrow Y$ is Lipschitz compact if, and only if, its associated linear operator $\widehat{F}\colon \mathcal{F}(M) \longrightarrow Y$ is compact, a fact that we will use without further mention. Therefore, if $Y$ is finite-dimensional then every Lipschitz map from $M$ to $Y$ is Lipschitz compact. We denote by $\SA_K(M,Y)$ the set of those Lipschitz compact maps from $M$ to $Y$ which strongly attain their norm, that is,
\[ \SA_K(M,Y)=\SA(M,Y)\cap\Lipc(M,Y).\]
Using the above notation, we are also interested in studying for which pointed metric spaces $M$ and Banach spaces $Y$, either the set $\SA_K(M,Y)$ is dense in $\Lipc(M,Y)$ or $(M,Y)$ has the Lip-BPB property for Lipschitz compact maps.

As we have said, all these questions have been studied before in the papers  \cite{articulo1,SNA_2,ChiMar-NA,lppr,Godefroy-survey-2015}, among others. Let us present some of the known results concerning these types of density. First, it was shown in \cite[Proposition 7.4]{lppr}, extending results from \cite{Godefroy-survey-2015}, that if $M$ is a pointed metric space so that $\mathcal{F}(M)$ has the Radon-Nikodym property (RNP), then $\SA(M,Y)$ is dense in $\Lip(M,Y)$ for every Banach space $Y$. Some cases of metric spaces $M$ for which $\mathcal{F}(M)$ has the RNP are: Uniformly discrete, countable and compact, and compact H\"{o}lder metric spaces (see Example 1.2 in \cite{articulo1} to get references of each result). Furthermore, in the paper \cite{articulo1} were given different properties over $\mathcal{F}(M)$ from the RNP that also ensures strong density for every Banach space. More precisely, they show that if $\mathcal{F}(M)$  satisfies either property $\alpha$ or property quasi-$\alpha$, then $\SA(M,Y)$ is dense in $\Lip(M,Y)$ for every Banach space $Y$. Also, we get the same result if we can find inside $\mathcal{F}(M)$ a uniformly strongly exposed $1$-norming set of molecules, that is, a uniformly strongly exposed set generating the unit ball of $\mathcal{F}(M)$ by closed absolutely convex hull. We refer to \cite[Section 3]{articulo1}, where these sufficient conditions and the relationship between them are discussed. Also, we refer to \cite[Theorem 2.5]{SNA_2}, where it is shown that the Radon-Nikodym property and having a uniformly strongly exposed $1$-norming set of molecules are not necessary conditions to obtain strong density, and they all are distinct.

On the other hand, let us comment some negative results. Recall that \cite[Example 2.1]{kms} proves that for $M=[0,1]$ with its usual metric and $Y=\mathbb{R}$, we have that $\SA([0,1],\mathbb{R})$ is not dense in $\Lip([0,1],\mathbb{R})$. In fact, they proved the same result for geodesic spaces. After that, it was generalized for length spaces (see Theorem 2.2 in \cite{articulo1}), where the lackness of strongly exposed points of $B_{\mathcal{F}(M)}$ seems to be essential in the proof of these results. However, if we consider $\mathbb{T}=\{x \in \mathbb{R}^2 \colon \|x\|_2=1\}$ endowed with the Euclidean metric, then the unit ball of $\mathcal{F}(\mathbb{T})$ is generated by its strongly exposed molecules (in fact, every molecule is strongly exposed), but Theorem 2.1 in \cite{SNA_2} shows that $\SA(\mathbb{T},\mathbb{R})$ is not dense in $\Lip(\mathbb{T},\mathbb{R})$.

Given a metric space $M$, the last cited result shows that the fact that $B_{\mathcal{F}(M)}$ must be generated by its extreme molecules, or even its strongly exposed molecules, does not guarantee to have density. As an interesting fact, Theorem 3.3 in \cite{SNA_2} proves that if $M$ is a metric space for which $\SA(M,\mathbb{R})$ is dense in $\Lip(M,\mathbb{R})$, then $B_{\mathcal{F}(M)}$ is generated by its extreme molecules, so the reciprocal result works. Moreover, if we assume that $M$ is compact, then Theorem 3.15 in \cite{SNA_2} shows that $B_{\mathcal{F}(M)}$ is indeed generated by its strongly exposed molecules.

With respect to the Lip-BPB property, given a finite pointed metric space $M$ and a Banach space $Y$, \cite[Theorem 2.1]{ChiMar-NA} shows that if $(\mathcal{F}(M),Y)$ has the classical BPBp, then $(M,Y)$ has the Lip-BPB property. This shows, for instance, that if $M$ is finite and $Y$ is finite-dimensional, then $(M,Y)$ has the Lip-BPB property. Moreover, Example 2.5 and Example 2.6 in \cite{ChiMar-NA} show that both finiteness assumptions are needed. It is also proved in that paper that if $M$ is a pointed metric space such that $\Mol(M)$ is a uniformly strongly exposed set, then $(M,Y)$ has the Lip-BPB property for any Banach space $Y$. This is the case when $M$ is concave and $\mathcal{F}(M)$ has property $\alpha$, $M$ is ultrametric, or $M$ is a H\"older metric space, among others.

As another section of the aforementioned paper \cite{ChiMar-NA}, we may also find some results focusing on the vector-valued case and on Lipschitz compact maps. More concretely, \cite[Proposition 4.4]{ChiMar-NA} states that if $M$ is a metric space such that $(M,\mathbb{R})$ has the Lip-BPB property and $Y$ is a Banach space satisfying property $\beta$, then $(M,Y)$ also has the Lip-BPB property. It is also shown that this result can be generalized in the case of the strong density by assuming that $Y$ has property quasi-$\beta$ instead. After that, analogous versions of the results just commented are given for the case of Lipschitz compact maps. Moreover, we can find some results only valid in this context. For instance, Proposition 4.17 in \cite{ChiMar-NA} says that if $M$ is a metric space such that $(M,\mathbb{R})$ has the Lip-BPB property and $Y^*$ is isometrically isomorphic to an $\operatorname{L}_1$-space, then $(M,Y)$ has the Lip-BPB property for Lipschitz compact maps.

After this background, we proceed to outline the content of this paper. The main aim of this work is to study the behavior under some natural operations on the domain or on the range spaces of the density of strongly norm attaining Lipschitz maps, of the Lip-BPB property, and of the corresponding versions for Lipschitz compact maps. The results will complement the study initiated in \cite{ChiMar-NA}.

Let us summarize the results of the paper by sections.

In the second section of this paper we focus our attention on the domain space, studying the behavior of the properties under sums of metric spaces (this sum can be somehow understood as the $\ell_1$-sum of the metric spaces after identifying their centers and can be found in \cite[Definition~1.13]{wea2}). First, Proposition \ref{domain} shows that if $M$ is the sum of two metric spaces $M_1$ and $M_2$ and $Y$ is a Banach space such that $(M,Y)$ has the LipBPB property, then so does $(M_i,Y)$ for $i=1,2$, but the reciprocal result is not true, as shown in Example \ref{example:M1M2no}. For the density of norm attaining Lipschitz maps, the situation is better: if $M$ is the (metric) sum of a family $\{M_i\colon i\in I\}$ of metric spaces and $Y$ is a Banach space, then $\SA(M,Y)$ is dense in $\Lip(M,Y)$ if and only if $\SA(M_i,Y)$ is dense in $\Lip(M_i,Y)$ for every $i\in I$ (Theorem \ref{l1SNA}).
We also obtain analogous results for Lipschitz compact maps.

Next, along the third section we study the problem of having strong density or the Lip-BPB property for vector-valued Lipschitz maps for some Banach space $Y$ assuming that we have such property for scalar-valued Lipschitz functions. Notice that Proposition 4.1 in \cite{ChiMar-NA} shows that if $M$ is a pointed metric space such that $(M,Y)$ has the Lip-BPB property for some Banach space $Y\neq 0$, then $(M,\mathbb{R})$ has the Lip-BPB property. Moreover, Proposition 4.2 in \cite{ChiMar-NA} shows an analogous result for the case of strong density. Therefore, in order to get the vector-valued versions of the last properties, these assumptions are necessary. Let us comment that it is not known if the assumption that $\SA(M,\mathbb{R})$ is dense in $\Lip(M,\mathbb{R})$ is also sufficient to get that $\SA(M,Y)$ is dense in $\Lip(M,Y)$ for every Banach space $Y$. However, Example 2.5 in \cite{ChiMar-NA} shows that the reciprocal result is not true in the case of the Lip-BPB property. In order to get positive results, we follow the recent paper \cite{ACK}, where the class of $\Gamma$\textit{-flat operators} and the notion of $ACK$ \textit{structure} on Banach spaces are introduced. Our main results in this section are Theorem \ref{teoACK} and Theorem \ref{ACKden}, from which we extract several implications about spaces of continuous functions, spaces of sequences, and injective tensor products. Also, we state the analogous versions of these theorems for Lipschitz compact maps, for which we obtain more satisfactory results. In this case, we also get some more results which complements those given in \cite{ChiMar-NA}.

Finally, in the forth section we focus our attention on stability properties related to the range space. First, given a pointed metric space $M$, Proposition \ref{recdomain} shows that if $Y_1$ is an absolute summand of a Banach space $Y$ and $(M,Y)$ has the Lip-BPB property (respectively, $\SA(M,Y)$ is dense), then so does $(M,Y_1)$ (resp.\ $\SA(M,Y_1)$ is dense). Reciprocally, Propositions \ref{Lip-BPBp domain} and \ref{c_0den} show that if $Y$ is a $c_0$ or $\ell_\infty$ sum of a family $\{Y_i\colon i\in I\}$ of Banach spaces such that all the pairs $(M,Y_i)$ have the Lip-BPB property with the same function $\eps\longmapsto \eta(\eps)$ (respectively, $\SA(M,Y_i)$ is dense in $\Lip(M,Y_i)$ for all $i\in I$), then $(M,Y)$ has the Lip-BPB property (respectively, $\SA(M,Y)$ is dense in $\Lip(M,Y)$). Furthermore, we also study some stability results involving continuous functions spaces. The corresponding results for Lipschitz compact maps are also stablished.

\section{Results concerning the domain space: relationship with metric sums}

In this section we will study the behavior of the Lip-BPB property with respect to ``metric" sums on the domain of Lipschitz maps. We need the following definition.

\begin{defi}[\textrm{\cite[Definition~1.13]{wea2}}]\label{l1suma}
Given a family of pointed metric spaces $\{(M_i,d_i)\}_{i\in I}$, the (metric) \emph{sum} of the family is the disjoint union of all $M_i$'s, identifying the base points, endowed with the following metric $d$: $d(x,y)=d_i(x,y)$ if both $x,y\in M_i$, and $d(x,y)=d_i(x,0)+d_j(0,y)$ if $x\in M_i$, $y\in M_j$ and $i\neq j$. We will write $\coprod_{i \in I} M_i$ to denote the sum of the family of metric spaces.
\end{defi}

It is known (c.f.\ e.g.\ \cite{kaufmannPreprint} or \cite[Proposition 3.9]{wea2}) that Lipschitz-free spaces behave well with respect to sums of metric spaces. Indeed, if $M=\coprod_{i \in I} M_i$, then
	\[ \mathcal{F}(M) \cong \left [\bigoplus\nolimits_{i \in I} \mathcal{F}(M_i)\right ]_{\ell_1} \]
	isometrically.

Now, we will present some results which show the good behavior of sums of metric spaces with respect to the Lip-BPB property and the density of $\SA(M,Y)$.
		
		\begin{prop}\label{domain}
			Let $M=M_1\coprod M_2$ be the sum of two pointed metric spaces and let $Y$ be a Banach space. If the pair $(M,Y)$ has the Lip-BPB property, then so does $(M_1,Y)$ and $(M_2,Y)$.
		\end{prop}
		
		\begin{proof}
			Fix $0<\varepsilon<1$ and let $\eta(\varepsilon)$ be the constant given by the Lip-BPB property of $(M,Y)$, which we may suppose that satisfies $\eta(\varepsilon)<\eps$.  Let $\widehat{F}_1 \in \operatorname{L}(\mathcal{F}(M_1),Y)$ with $\|F_1\|_L=1$ and $m \in \Mol(M_1)$ such that $\|\widehat{F}_1(m)\|>1-\eta(\varepsilon)$. Now, let us define $\widehat{F} \in \operatorname{L}(\mathcal{F}(M),Y)$ by
			\[
			F(p) = \left\{
			\begin{array}{@{}l@{\thinspace}l}
			F_1(p) &\text{ if } p\in M_1,\\
			0 &\text{ if } p \in M_2.\\
			
			\end{array}
			\right.
			\]
			It is easy to see that $\|F\|_L=1$ and $\|\widehat{F}(m)\|>1-\eta(\varepsilon)$, where we see $m$ as a molecule of $\mathcal{F}(M)$. By hypothesis, there exist $\widehat{G} \in \operatorname{L}(\mathcal{F}(M),Y)$ and a molecule $u \in \Mol(M)$ such that
			\[ \|\widehat{G}(u)\|=\|G\|_L=1,\quad\|m-u\|<\varepsilon,\quad\|F-G\|_L<\varepsilon.\]
			Consider $\widehat{G}_1 \in \operatorname{L}(\mathcal{F}(M_1),Y)$ to be the restriction of $\widehat{G}$ to the subspace $\mathcal{F}(M_1)$. Then, it is clear that $\|G_1\|_L\leq \|G\|_L= 1$ and $\|F_1-G_1\|_L\leq \|F-G\|_L<\varepsilon$. Hence, it will be enough to show that $\widehat{G}_1$ attains its norm at a molecule close enough to $m$. Let us write
			\[ u=\frac{\delta_p-\delta_q}{d(p,q)},\]
			where $p$, $q \in M$, $p\neq q$. We distinguish four cases:
			\begin{enumerate}
				\item $p, q \in M_1$:
				In this case  $u$ can be seen as a molecule of $M_1$ and so $\widehat{G}_1$ attains its norm at $u$.
				\item $p, q \in M_2$: Then, note that
				\[ \widehat{F}(u)=\frac{F(p)-F(q)}{d(p,q)}=0,\]
				from where we deduce that $\|\widehat{G}(u)\|<\varepsilon$, a contradiction.
				\item $p \in M_1, q \in M_2$: Let us write $u$ as the following convex combination:
				\[ u=\frac{\delta_p-\delta_q}{d(p,q)}=\frac{\delta_p-\delta_0}{d(p,0)}\frac{d(p,0)}{d(p,q)} + \frac{\delta_0-\delta_q}{d(0,q)}\frac{d(0,q)}{d(p,q)}=m_{p,0}\frac{d(p,0)}{d(p,q)} + m_{0,q}\frac{d(0,q)}{d(p,q)}. \]
				Since $\widehat{G}$ attains its norm at $u$, then it also attains its norm at $m_{p,0} \in \Mol(M_1)$. Hence, $\widehat{G}_1$ attains its norm at  $m_{p,0}$. Also, note that
				\[ \|\widehat{F}(u)\|=\frac{d(p,0)}{d(p,q)}\|\widehat{F}(m_{p,0})\|\leq\frac{d(p,0)}{d(p,q)}.\]
				On the other hand, $\|\widehat{F}(m)\|>1-\eta(\varepsilon)$ and $\|m-u\|<\varepsilon$. Therefore, we must have that $\|\widehat{F}(u)\|>1-\eta(\varepsilon)-\varepsilon$, from where $\frac{d(p,0)}{d(p,q)}>1-\eta(\varepsilon)-\varepsilon$. Consequently, $\frac{d(0,q)}{d(p,q)}<\eta(\varepsilon)+\varepsilon$. Now, note that
				\begin{align*}
				\|m-m_{p,0}\|&=\left \|(m-u)+\left (\frac{d(p,0)}{d(p,q)}-1\right )m_{p,0} + \frac{d(0,q)}{d(p,q)}m_{0,q}\right \|\\&\leq \|m-u\|+2\frac{d(0,q)}{d(p,q)}\leq \|m-u\|+2\eta(\varepsilon)+2\varepsilon<2\eta(\varepsilon)+3\varepsilon<5\varepsilon.
				\end{align*}
				\item $p \in M_2, q \in M_1$: We just have to repeat the previous argument.
			\end{enumerate}
			Consequently, we conclude that $(M_1,Y)$ has the Lip-BPB property. Since the situation is symmetric, we also get that $(M_2,Y)$ has the Lip-BPB property.
		\end{proof}
	
Note that from this result we obtain the next corollary by just observing that for every $j \in I$, we have that $\coprod_{i\in I} M_i \equiv M_j \coprod Z$
for some pointed metric space $Z$.
	
	\begin{cor}
		Let $M=\coprod_{i \in I} M_i$ be the sum of a family $\{M_i\}_{i\in I}$ of pointed metric spaces and let $Y$ be a Banach space. If the pair $(M,Y)$ has the Lip-BPB property, then so does $(M_i,Y)$ for every $i \in I$.
	\end{cor}
	
	The converse result of Proposition \ref{domain} is false, as the next example shows.
\begin{ejem}\label{example:M1M2no}
Let $M_1=\{0,1\}$ and $M_2=\{1,2\}$ viewed as subsets of $\R$ with the usual metric and consider $1$ as base point for both spaces. First, observe that $M=M_1\coprod M_2$ is isometric to the subset $\{0,1,2\}$ of $\R$ with the usual metric.
Now, the pairs $(M_i,Y)$ has the Lip-BPB property for $i=1,2$ and every Banach space $Y$ (this is obvious as the spaces $\mathcal{F}(M_1)$ and $\mathcal{F}(M_2)$ are one-dimensional), but for every strictly convex Banach space $Y$ which is not uniformly convex, the pair $(M,Y)$ fails the Lip-BPB property, see \cite[Example~2.5]{ChiMar-NA}.
\end{ejem}
	
In the case of the density of $\SA(M,Y)$, we actually get a characterization, as the following result shows.
	
	\begin{teo}\label{l1SNA}
		Let $\{M_i\}_{i \in I}$ be a family of pointed metric spaces, consider the sum $M=\coprod_{i\in I} M_i$ and let $Y$ be a Banach space. Then the following are equivalent:
		\begin{enumerate}[(1)]
			\item $\SA(M_i,Y)$ is dense in $\Lip(M_i,Y)$ for every $i \in I$.
			\item $\SA(M,Y)$ is dense in $\Lip(M,Y)$.
		\end{enumerate}
	\end{teo}
	
	\begin{proof}
		(1) $\Rightarrow$ (2)
		Let us consider the natural embeddings $E_i \colon \mathcal{F}(M_i) \longrightarrow \mathcal{F}(M)$ and the natural projections $P_i \colon \mathcal{F}(M) \longrightarrow \mathcal{F}(M_i)$ for every $i \in {I}$. Fix $\varepsilon>0$ and take $\widehat{F} \in L(\mathcal{F}(M),Y)\cong \Lip(M,Y)$. Without loss of generality, we may assume that $\| F \|_L=1$. Using that $\| F\|_L=\sup\{ \|\widehat{F} E_i\| \colon i \in I\}$ we can find $h \in I$ such that
		$ \|\widehat{F}E_h\|> \|F\|_L -\varepsilon$. By hypothesis, we can find $G_h \in \SA(M_h,Y)$ verifying $\|G_h\|_L= \|\widehat{F}E_h\|$ and $\|\widehat{G}_h-\widehat{F}E_h\|\leq \varepsilon$. Let us define $\widehat{G} \in \operatorname{L}(\mathcal{F}(M),Y)$ by
		\[ \widehat{G}E_i=(1-\varepsilon) \widehat{F}E_i\, \mbox{ for } i \in I, i \neq h \quad  \mbox{ and }\quad  \widehat{G}E_h=\widehat{G}_h. \]
		Then, $\|G\|_L=\sup\{\|\widehat{G}E_i\|\colon i \in I\}=\|G_h\|_L $ and $\| G-F\|_L=\sup\{\|(\widehat{G}-\widehat{F})E_i\|\colon i \in I\} \leq \varepsilon$. Moreover, note that if we take a molecule $m_{p_h,q_h} \in \mathcal{F}(M_h)$ such that $ \|\widehat{G}_h(m_{p_h,q_h})\| =\|G_h\|_L$ then, if we consider the molecule $E_h(m_{p_h,q_h}) \in \mathcal{F}(M)$, we will have that $ \|\widehat{G}(E_h(m_{p_h,q_h}))\|  = \| \widehat{G}_h (m_{p_h,q_h}) \|=\|G_h\|_L=\|G\|_L$. Hence, $G \in \SA(M,Y)$.

		(2) $\Rightarrow$ (1) Fix $\varepsilon >0$, $h \in I$ and take $\widehat{F}_h \in \Lip(M_h,Y)$. As above, we may assume that $\|F_h\|_L=1$. Let us define $\widehat{F} \colon \mathcal{F}(M) \longrightarrow Y$ by $\widehat{F}=\widehat{F}_hP_h$. Then, it is clear that $\| F\|_L=\|F_h\|_L=1$. By hypothesis, we can find $G \in \Lip(M,Y)$ such that $\|G\|_L=1$ and $\|G-F\|_L\leq \varepsilon$. Now, define $\widehat{G}_h \colon \mathcal{F}(M_h)\longrightarrow Y$ by $\widehat{G}_h=\widehat{G}E_h$. Then, $\|G_h\|_L\leq 1$ and $\|\widehat{G}_h-\widehat{F}_h\|=\|\widehat{G}E_h-\widehat{F}E_h\|\leq \|G-F\|_L\leq \varepsilon$, so we just have to see that $G_h \in \SA(M_h,Y)$. To this end, consider a molecule $m_{p,q} \in \mathcal{F}(M)$ such that $\| \widehat{G}(m_{p,q}) \| =\|G\|_L=1$. We claim that $P_h(m_{p,q})$ is a molecule of $\mathcal{F}(M_h)$. Then, we would have that $$\| \widehat{G}_h(P_h(m_{p,q})) \|=\| \widehat{G}(m_{p,q}) \| = \|G\|_L=\|G_h\|_L=1.$$
		Hence, $G_h \in \SA(M_h,Y)$ and the result would be proved. Indeed, assume that $P_h(m_{p,q})$ is not a molecule of $\mathcal{F}(M_h)$. Then, either $p\notin M_h$ or $q \notin M_h$. If we assume $q \notin M_h$, we will have that $P_h(\delta_q)=0$, but $\widehat{G}_h$ attains its norm at $P_h(m_{p,q})$, so $P_h(m_{p,q})\neq0$, which implies that $p \in M_h$. Finally, observe that
		\[\| \widehat{G}_h(P_h(m_{p,q})) \|=\frac{\widehat{G}_h(P_h(\delta_p))-\widehat{G}_h(P_h(\delta_q))}
		{d(p,q)}=\frac{\widehat{G}_h(\delta_p)-\widehat{G}_h(\delta_0)}{d(p,q)}<\frac{\widehat{G}_h(\delta_p)-\widehat{G}_h(\delta_0)}{d(p,0)} \leq \|G_h\|_L, \]
		a contradiction. The case $p \notin M_h$ is analogous to the above one.
	\end{proof}

Let us show that the previous results also work for Lipschitz compact maps.	

	\begin{prop}\label{domaincompact}
		Let $M=M_1\coprod M_2$ be the sum of two pointed metric spaces and let $Y$ be a Banach space. If the pair $(M,Y)$ has the Lip-BPB property for Lipschitz compact maps, then so do $(M_1,Y)$ and $(M_2,Y)$.
	\end{prop}

	\begin{proof}
		The result follows by repeating the proof of Proposition \ref{domain} for a Lipschitz compact map $F_1$ observing that, in such case, the strongly norm attaining Lipschitz map which approximates $F_1$ is Lipschitz compact too.
	\end{proof}

	From this result we obtain the following corollary.
	
	\begin{cor}
		Let $M=\coprod_{i \in I} M_i$ be the sum of pointed metric spaces and let $Y$ be a Banach space. If the pair $(M,Y)$ has the Lip-BPB property for Lipschitz compact maps, then so does $(M_i,Y)$ for every $i \in I$.
	\end{cor}

	In the same way as in the general case, the converse of Proposition \ref{domaincompact} is not true, as the same Example \ref{example:M1M2no} shows. And again, the analogous result for the density of $\SA_K(M,Y)$ is more satisfactory.
	
	\begin{prop}
			Let $\{M_i\}_{i \in I}$ be a family of pointed metric spaces, let $Y$ be a Banach space, and consider $M=\coprod_{i\in I} M_i$. Then the following are equivalent:
		\begin{enumerate}[(1)]
			\item $\SA_K(M_i,Y)$ is dense in $\Lipc(M_i,Y)$ for every $i \in I$.
			\item $\SA_K(M,Y)$ is dense in $\Lipc(M,Y)$.
		\end{enumerate}
	\end{prop}

	\begin{proof}
		It is enough to note that the operators $\widehat{G}$ and $\widehat{G}_h$ defined in the proof of Theorem \ref{l1SNA} are compact when the operators $\widehat{F}$ and $\widehat{F}_h$ are compact.
	\end{proof}

\section{From scalar functions to vector-valued maps}

	Our aim in this section is to study the problem of passing from the Lip-BPB property for scalar-valued functions to some vector-valued maps and the problem of passing from the density of $\SA(M,\R)$ to the density of $\SA(M,Y)$ for some $Y$'s. Let us comment that one can find in \cite{ChiMar-NA} examples of metric spaces $M$ for which there are Banach spaces $Y$ so that the pairs $(M,\R)$ have the Lip-BPB property, but the pairs $(M,Y)$ fail to have the Lip-BPB property. Actually, our Example~\ref{example:M1M2no} contains one of such metric spaces: $M=\{0,1,2\}$ with the usual metric coming from $\R$. We will present sufficient conditions on a Banach space $Y$ assuring that a pair $(M,Y)$ has the Lip-BPB property when $(M,\mathbb{R})$ does, extending \cite[Proposition~4.4]{ChiMar-NA} in which the result is proved for $Y$ having Lindenstrauss' property $\beta$. With respect to the density of strongly norm attaining Lipschitz maps, let us comment that we do not know of any metric space $M$ such that $\SA(M,\R)$ is dense in $\Lip(M,\R)$ but there is $Y$ such that $\SA(M,Y)$ is not dense in $\Lip(M,Y)$. Nevertheless, we will also present sufficient conditions on a Banach space $Y$ assuring that the density of $\SA(M,\R)$ implies the density of $\SA(M,Y)$.

Our work is based on the recent paper \cite{ACK}. First of all, we need to give the necessary notions.
	
	\begin{defi}
		Let $A$ be a topological space and $(M,d)$ be a metric space. A function $f\colon A \longrightarrow M$ is said to be \emph{openly fragmented}, if for every nonempty open subset $U \subset A$ and every $\varepsilon>0$ there exists a nonempty open subset $V \subset U$ with $\diam(f(V))<\varepsilon$.
	\end{defi}

	It is clear that every continuous function $f\colon A\longrightarrow M$ is openly fragmented. In particular, if $A$ is a discrete topological space then every $f \colon A \longrightarrow M$ is openly fragmented.
	
		\begin{defi}
		Let $X$, $Y$ be Banach spaces and $\Gamma \subset Y^*$. An operator $T \in \operatorname{L}(X,Y)$ is said to be $\Gamma$\emph{-flat}, if $T^*|_\Gamma \colon (\Gamma, \omega^*)\longrightarrow (X^*, \| \cdot \|_{X^*})$ is openly fragmented. In other words, if for every $\omega^*$-open subset $U\subseteq Y^*$ with $U\cap\Gamma \neq \emptyset$ and every $\varepsilon>0$ there exists a $w^*$-open subset $V \subset U$ with $V\cap \Gamma \neq \emptyset$ such that $\diam(T^*(V\cap\Gamma))<\varepsilon$. The set of all $\Gamma$-flat operators in $\operatorname{L}(X,Y)$ will be denoted by $\operatorname{Fl_\Gamma}(X,Y)$.
	\end{defi}

	In \cite{ACK} it is shown that every Asplund operator $T \in \operatorname{L}(X,Y)$ is $\Gamma$-flat for every $\Gamma \subseteq B_{Y^*}$. Consequently, every compact operator is $\Gamma$-flat for every $\Gamma \subseteq B_{Y^*}$. In addition, it is shown that if $(\Gamma, \omega^*)$ is discrete then every bounded operator $T \in \operatorname{L}(X,Y)$ is $\Gamma$-flat.

	Finally, they introduce the notion of $ACK_\rho$ structure, which has the structural properties of $C(K)$ and its uniform subalgebras that are essential for the BPB property to hold. Let us recall that a subset $\Gamma$ of the unit ball of the dual of a Banach space $Y$ is \emph{$1$-norming} if the absolutely weak-star closed convex hull of $\Gamma$ equals the whole of $B_{Y^*}$ or, equivalently, if $\|y\|=\sup\{|f(y)|\colon f \in \Gamma\}$ for every $y \in Y$.

	\begin{defi}\label{ACK}
We say that a Banach space $Y$ has $ACK$ \emph{structure} with parameter $\rho$, for some $\rho \in [0,1)$ ($Y \in ACK_\rho$ for short) whenever there exists a $1$-norming set $\Gamma \subset B_{Y^*}$ such that for every $\varepsilon>0$ and every nonempty relatively $\omega^*$-open subset $U\subset \Gamma$ there exist a nonempty subset $V\subset U$, vectors $y_1^*\in V$, $e\in S_X$, and an operator $F \in \operatorname{L}(Y,Y)$ with the following properties:
		\begin{enumerate}
			\item $\|Fe\|=\|F\|=1$;
			\item $y_1^*(Fe)=1$;
			\item $F^*y_1^*=y_1^*$;
			\item denoting $V_1=\{y^* \in \Gamma \colon \|F^*y^*\|+(1-\varepsilon)\|(\operatorname{I}_{Y^*} - F^*)(y^*)\|\leq1\}$, then $|v^*(Fe)|\leq \rho$ for every $v^* \in \Gamma \setminus V_1$;
			\item $d(F^*y^*,\operatorname{aco}\{0,V\})<\varepsilon$ for every $y^* \in \Gamma$; and
			\item $|v^*(e)-1|\leq \varepsilon$ for every $v^* \in V$.
		\end{enumerate}
The Banach space $Y$ has \emph{simple $ACK$ structure} ($X \in ACK$) if $V_1=\Gamma$ (and so $\rho$ is redundant).
\end{defi}

The following statement is a compilation of results that can be found in \cite{ACK}. We introduce some notation. Given a Banach space $Y$, we write $c_0(Y,w)$ to denote the Banach space of all weakly null sequences in $Y$; if $K$ is a compact Hausdorff topological space, $C_w(K,Y)$ is the Banach space of all $Y$-valued weakly continuous functions from $K$ to $Y$.

	\begin{prop}[\cite{ACK}]\label{propACK}  The following statements hold.
		\begin{enumerate}
			\item $C(K)$ has simple $ACK$ structure for every compact Hausdorff topological space $K$.
			\item Finite injective tensor products of Banach spaces which have $ACK_\rho$ structure also have $ACK_\rho$ structure.
			\item Given a compact Hausdorff topological space $K$, if $Y \in ACK_\rho$  then $C(K,Y) \in ACK_\rho$.
			\item Let $Y$ be a Banach space having $ACK_\rho$ structure. Then $c_0(Y)$, $\ell_\infty(Y)$, and $c_0(Y,w)$ have $ACK_\rho$ structure.
			\item Given a compact Hausdorff topological space $K$, if $Y \in ACK_\rho$, then $C_w(K,Y)$ has $ACK_\rho$ structure.
		\end{enumerate}
	\end{prop}	

The main result of this section is the following.
	
	\begin{teo}\label{teoACK}
		Let $M$ be a pointed metric space such that $(M,\mathbb{R})$ has the Lip-BPB property, let $Y$ be a Banach space in $ACK_\rho$ with associated $1$-norming set $\Gamma \subseteq B_{Y^*}$ of Definition \ref{ACK}, and let $\varepsilon>0$. Then, there exists $\eta(\varepsilon,\rho)>0$ such that if we take $\widehat{T} \in \operatorname{L}(\mathcal{F}(M),Y)$ a $\Gamma$-flat operator with $\|T\|_L=1$ and $m \in \Mol(M)$ satisfying $\|\widehat{T}(m)\|>1-\eta(\varepsilon,\rho)$, then there exist an operator $\widehat{S} \in \operatorname{L}(\F(M),Y)$ and a molecule $u \in \Mol(M)$ such that
		\[ \|\widehat{S}(u)\|=\|S\|_L=1, \quad \|m-u\|<\varepsilon, \quad \|T-S\|_L<\varepsilon. \]
	\end{teo}

Prior to give the proof of the theorem, we present the main consequences of Theorem \ref{teoACK}.

\begin{cor}\label{corollary:all_consequences_ACKrho}
	Let $M$ be a pointed metric space such that $(M,\mathbb{R})$ has the Lip-BPB property. The following statements hold.
	\begin{enumerate}
		\item For every compact Hausdorff topological space $K$, the pair $(M,C(K))$ has the Lip-BPB property for $\Gamma$-flat operators, where $\Gamma$ is the $1$-norming set given by Definition \ref{ACK} for $C(K)$.
		\item Let $Z$ be a finite injective tensor product of Banach spaces which have $ACK_\rho$ structure. Then, $(M,Z)$ has the Lip-BPB property for $\Gamma$-flat operators, where $\Gamma$ is the $1$-norming set given by Definition \ref{ACK} for $Z$.
		\item Let $K$ be a compact Hausdorff topological space. If $Y \in ACK_\rho$, then $(M,C(K,Y))$ and $(M,C_w(K,Y))$ have the Lip-BPB property for $\Gamma$-flat operators, where $\Gamma$ is the $1$-norming set given by Definition \ref{ACK} for $C(K,Y)$ and $(M,C_w(K,Y))$, respectively.
		\item Let $Y \in ACK_\rho$. Then, $(M,c_0(Y))$, $(M,\ell_\infty(Y))$, and $(M,c_0(Y,w))$ have the Lip-BPB property for $\Gamma$-flat operators, where $\Gamma$ is the corresponding $1$-norming set given by Definition \ref{ACK} for each case.
	\end{enumerate}
\end{cor}

The proof follows immediately from Theorem \ref{teoACK} and Proposition \ref{propACK}.

\begin{remark}\label{remark:Gamma-C(K)}
Let us give some comments on the assertion (1) of Corollary~\ref{corollary:all_consequences_ACKrho}. First, the set $\Gamma$ of Definition \ref{ACK} for the case $Y=C(K)$ is just $\Gamma=\{\delta_t\colon t\in K\}\subset S_{C(K)^*}$ (this follows from the results in the paper \cite{ACK}), so given $T\in \operatorname{L}(X,C(K))$, $T^*|_\Gamma$ is just the usual representation function of the operator $T$, that is, $\mu_T\colon K\longrightarrow X^*$ given by $\mu_T(t)=T^*(\delta_t)$ for all $t\in K$. This procedure actually gives an identification between $\operatorname{L}(X,C(K))$ and the space of those weak-star continuous functions $\mu\colon K\longrightarrow X^*$. Norm continuous functions correspond to compact operators (which are $\Gamma$-flat). We do not know which functions are openly fragmented or, equivalently, which functions correspond to $\Gamma$-flat operators, but there is an intermediate condition which has been studied widely in the literature: quasi-continuous functions. A function $\mu\colon K\longrightarrow X^*$ is \emph{quasi-continuous} if for every non-empty open subset $U\subset K$, every $s\in U$, and every neighborhood $V$ of $\mu(s)$, there exists a non-empty open subset $W\subset U$ such that $\mu(W)\subset V$. This is a classical notion which is still investigated, see the paper \cite{Banakh} and references therein for a sample. Quasi-continuous functions are openly fragmented and they form a class more general than the one of continuous functions.
\end{remark}

Let us comment that there is one more consequence of Theorem~\ref{teoACK} that was already stated in \cite[Proposition~4.4]{ChiMar-NA} with a different proof. Indeed, if a Banach space $Y$ has Lindenstrauss' property $\beta$ (see \cite[Definition 4.3]{ChiMar-NA} for instance), then $Y\in ACK_\rho$ for a discrete $1$-norming set $\Gamma$, so every operator arriving to $Y$ is $\Gamma$-flat. Therefore, it follows from Theorem \ref{teoACK} another proof of the following fact given in \cite[Proposition~4.4]{ChiMar-NA}: if $(M,\R)$ has the Lip-BPB property and $Y$ has property $\beta$, then $(M,Y)$ has the Lip-BPB property.

Let us now prepare the way for the proof of Theorem \ref{teoACK} by presenting some preliminary results.
	
	\begin{lem}\label{Lema1}
		Let $M$ be a pointed metric space and let $\varepsilon>0$. Suppose that $(M, \mathbb{R})$ has the Lip-BPB property witnessed by the function $\varepsilon\longmapsto \eta(\varepsilon)>0$. Then, given $f \in \Lip(M,\mathbb{R})$ with $\|f\|_L\leq 1$ and $m \in \Mol(M)$ such that $|\widehat{f}(m)|> 1-\eta(\varepsilon)$, there exist $g \in \Lip(M,\mathbb{R})$ with $\|g\|_L=1$ and $u \in \Mol(M)$ satisfying
		\[ |\widehat{g}(u)|=1, \quad \|f-g\|_L< \varepsilon+\eta(\varepsilon), \quad \|m-u\|< \varepsilon. \]
	\end{lem}
	
	\begin{proof}
		If $\|f\|_L=1$ then it is enough to apply the Lip-BPB property. If $\|f\|_L<1$, by applying the Lip-BPB property, we know that there exist $g \in S_{\Lip(M,\mathbb{R})}$ and $u \in \Mol(M)$ satisfying
		\[ \left\|g-\frac{f}{\|f\|}\right\|_L <\varepsilon, \quad \|u-m\|< \varepsilon. \]
		Then, note that
		\[ \|g-f\|_L\leq \left\|g-\frac{f}{\|f\|}\right\|_L+\left\|\frac{f}{\|f\|}-f\right\|_L<\varepsilon+|1-\|f\|_L|\leq \varepsilon+\eta(\varepsilon).\qedhere\]
	\end{proof}

	\begin{lem}\label{Lema2}
		Let $M$ be a pointed metric space such that $(M,\mathbb{R})$ has the Lip-BPB property, let $Y$ be a Banach space, and let $\Gamma \subseteq B_{Y^*}$ be a $1$-norming set. Fix $\varepsilon>0$ and consider $\eta(\varepsilon)$ the constant given by the Lip-BPB property of $(M,\mathbb{R})$. Let $\widehat{T} \in \operatorname{Fl}_\Gamma(\F(M),Y)$ be a $\Gamma$-flat operator with $\|T\|_L=1$ and $m \in \Mol(M)$ such that\[\|\widehat{T}(m)\|>1-\eta(\varepsilon).\]
		Then, for every $r>0$ there exist:
		\begin{enumerate}
			\item a $\omega^*$-open subset $U_r\subset V$ with $U_r \cap \Gamma \neq \emptyset$,
			\item $\widehat{f}_r \in S_{\mathcal{F}(M)^*}$ and $u_r \in \Mol(M)$ satisfying
			\[\widehat{f}_r(u_r)=1, \quad \|m-u_r\|\leq \varepsilon, \quad \|\widehat{T}^*z^* - \widehat{f}_r\|\leq r+\varepsilon+\eta(\varepsilon) \quad \forall \, z^* \in U_r \cap \Gamma. \]
		\end{enumerate}
	\end{lem}
	
	\begin{proof}
		We just have to repeat the proof of Lemma 2.9 in \cite{ACK} using Lemma \ref{Lema1} instead of Proposition 2.11 in \cite{ACK}.
	\end{proof}

Now, we are able to prove the main result of this section.
	
	\begin{proof}[Proof of Theorem \ref{teoACK}]
		Given $\varepsilon>0$, let $\widehat{\eta}(\varepsilon)>0$ be the constant associated to the Lip-BPB property of $(M,\mathbb{R})$. Fix $0<\varepsilon_0<\varepsilon$ and take $\varepsilon_1>0$ such that
		\[
\max\left\{\varepsilon_1, 2\left( (\varepsilon_1+\eta(\varepsilon_1))+ \frac{2(\varepsilon_1+\eta(\varepsilon_1))}{1-\rho+(\varepsilon_1+\eta(\varepsilon_1))}\right)\right\}\leq \varepsilon_0.
\]
		Take $r>0$ and $0<\varepsilon_2<\frac{2}{3}$. Consider $\widehat{T} \in \operatorname{L}(\mathcal{F}(M),Y)$ a $\Gamma$-flat operator with $\|T\|_L=1$ and a molecule $m \in \Mol(M)$ such that $\|\widehat{T}(m)\|>1-\widehat{\eta}(\varepsilon)$. Then, applying Lemma \ref{Lema2} with $Y$, $\Gamma$, $r$ and $\varepsilon_1$, we obtain an $\omega^*$-open subset $U_r\subseteq Y^*$ with $U_r\cap\Gamma \neq \emptyset$, and  $\widehat{f}_r \in S_{\mathcal{F}(M)^*}$, $u_r \in \Mol(M)$ satisfying
		\[ \widehat{f}_r(u_r)=1, \quad \|m-u\|\leq \varepsilon_1, \quad \|\widehat{T}^*z^*-\widehat{f}_r\|\leq r+\varepsilon_1+\eta(\varepsilon_1)\quad \forall \, z^* \in U_r\cap\Gamma.\]
		On the other hand, since $U_r\cap\Gamma \neq \emptyset$, by applying the definition of $ACK_\rho$ structure to $U=U_r\cap\Gamma$ and $\varepsilon_2$, we obtain a nonempty subset $V\subseteq U$, points $y_1^*\in V$ and $e \in S_{Y}$, an operator $F\in\operatorname{L}(Y,Y)$, and a subset $V_1\subseteq \Gamma$ satisfying the properties of Definition \ref{ACK}.

		Let us define the linear operator $\widehat{S}\colon \F(M) \longrightarrow Y$ by
		\[ \widehat{S}(x)=\widehat{f}_r(x)Fe+(1-\delta)(\operatorname{Id}_Y-F)\widehat{T}(x),\]
		where $\delta \in [\varepsilon_2,1)$. We will choose $\delta$ so that $\|\widehat{S}\|\leq 1$. In order to estimate $\|\widehat{S}\|$, recall that since $\Gamma$ is a $1$-norming set, we have that
		\[ \|\widehat{S}\|=\|\widehat{S}^*\|=\sup\left \{\bigl\|\widehat{S}^*y^*\bigr\|\colon y^* \in \Gamma \right \}.\]
		Therefore, we take $y^* \in \Gamma$ and estimate
		\[ \|\widehat{S}^* y^*\|= \bigl\|y^*(Fe)\widehat{f}_r + (1-\delta)\widehat{T}^*(\operatorname{Id}_{Y^*}-F^*)(y^*)\bigr\|. \]
		If $y^* \in V_1$, then that $\|\widehat{S}^*y^*\|\leq 1$ follows from the property (4) of Definition \ref{ACK}. Therefore, we only have to consider the case when $y^* \in \Gamma \setminus V_1$.
		As before, by Definition \ref{ACK}, for every $y^* \in \Gamma$ there exists a point $v^*=\sum_{k=1}^n \lambda_k v_k^*$ satisfying
		\[ \{v_1^*,\ldots,v_n^*\}\subseteq V, \quad \sum_{k=1}^n |\lambda_k|\leq 1, \quad \|F^*y^*-v^*\|<\varepsilon_2.\]
		Consequently,
		\begin{align*}
		\|v^*(e)\widehat{f}_r-\widehat{T}^*v^*\|&\leq\sum_{k=1}^n |\lambda_k|\|v_k^*(e)\widehat{f}_r -\widehat{T}^*v_k^*\|\\&\leq \sum_{k=1}^n |\lambda_k|(\|v_k^*(e)\widehat{f}_r-\widehat{f}_r\|+\|\widehat{f}_r-\widehat{T}^*v_k^*\|)\\&\leq \varepsilon_2 + \sum_{k=1}^n |\lambda_k| \|\widehat{f}_r-\widehat{T}^*v_k^*\| \leq \varepsilon_2+r+\varepsilon_1+\eta(\varepsilon_1).
		\end{align*}
		Now, for every $y^*\in \Gamma \setminus V_1$ we have that
		\begin{align*}
		\|\widehat{S}^*y^*\|&\leq \delta|y^*(Fe)|+(1-\delta)\|y^*(Fe)\widehat{f}_r+\widehat{T}^*y^*-\widehat{T}^*F^*y^*\|\\
		&\leq \delta\rho + (1-\delta)\|\widehat{T}^*y^*\|+(1-\delta)\|(F^*y^*)(e)\widehat{f}_r-\widehat{T}^*F^*y^*\|\\
		&\leq \delta\rho+(1-\delta)+2\varepsilon_2(1-\delta)+(1-\delta)\|v^*(e)\widehat{f}_r-R^*v^*\|\\
		&\leq \delta\rho + (1-\delta) + 2\varepsilon_2(1-\delta) + (1-\delta)(\varepsilon_2+r+\varepsilon_1+\eta(\varepsilon_1))\\
		&\leq \delta\rho+(1-\delta)(1+3\varepsilon_2+r+\varepsilon_1+\eta(\varepsilon_1)).
		\end{align*}
		Therefore, if we choose
		\[ \delta=\frac{3\varepsilon_2+r+\varepsilon_1+\eta(\varepsilon_1)}{1-\rho+3\varepsilon_2+r+\varepsilon_1+\eta(\varepsilon_1)} \in \left [2/3,1\right ) \subseteq [\varepsilon_2,1),\]
		then we will have that $\|\widehat{S}\|\leq 1$. In this case,
		\[1=|\widehat{f}_r(u)|=|y_1^*(\widehat{f}_r(u)Fe)|=|y_1^*(\widehat{S}(u))|\leq\|\widehat{S}(u)\|\leq 1,\]
		from which we deduce that $\|\widehat{S}\|=1$ and $\widehat{S}$ attains its norm at the molecule $u$, which we already knew satisfies $\|m-u\|\leq \varepsilon_1\leq \varepsilon_0<\varepsilon$.

		Finally, let us estimate $\|\widehat{S}-\widehat{T}\|$. First,
		\begin{align*}\|\widehat{S}-\widehat{T}\|&=\|\widehat{S}^*-\widehat{T}^*\|= \sup\bigl\{\bigl|\widehat{S}^*y^*-\widehat{T}^*y^*\bigr|\colon y^*\in \Gamma\bigr\}\\
		&\leq 2\delta+\sup\left \{\|y^*(Fe)\widehat{f}_r-\widehat{T}^*F^*y^*\|\colon y^*\in \Gamma\right \}.
		\end{align*}
		Second,
		\begin{align*}
		\|(F^*y^*)(e)\widehat{f}_r-\widehat{T}^*F^*y^*\|\leq 2\varepsilon_2+\|v^*(e)\widehat{f}_r-\widehat{T}^*v^*\|
		& \leq 3\varepsilon_2+r+\varepsilon_1+\eta(\varepsilon_1).
		\end{align*}
		Therefore, we obtain that
		\[\|\widehat{S}-\widehat{T}\|\leq 2\delta+3\varepsilon_2+r+\varepsilon_1+\eta(\varepsilon_1). \]
		Since $\varepsilon_2$ and $r$ were arbitrary, by taking these constants verifying $3\varepsilon_2+r\leq \varepsilon_1+\eta(\varepsilon_1)$, we will have that
		\[ \|\widehat{S}-\widehat{T}\|\leq 2(\varepsilon_1+\eta(\varepsilon_1)+\delta)\leq 2\left((\varepsilon_1+\eta(\varepsilon_1))+ \frac{2(\varepsilon_1+\eta(\varepsilon_1))}{1-\rho+\varepsilon_1+\eta(\varepsilon_1)}\right)\leq \varepsilon_0<\varepsilon. \qedhere\]
	\end{proof}

	Given a pointed metric space $M$ and a Banach space $Y$, it is possible to give a result analogous to Theorem \ref{teoACK} but for the density of $\SA(M,Y)$. We just have to repeat the previous proof using that $\SA(M,\mathbb{R})$ is dense in $\Lip(M,\mathbb{R})$ instead of the Lip-BPB property of $(M,\mathbb{R})$, forgetting the estimation on the distance between molecules.
	
	\begin{teo}\label{ACKden}
		Let $M$ be a pointed metric space such that $\SA(M,\mathbb{R})$ is dense in $\Lip(M,\mathbb{R})$, let $Y$ be a Banach space in $ACK_\rho$, and let $\Gamma \subseteq B_{Y^*}$ be the $1$-norming set given by Definition \ref{ACK}. Then, we have that
		\[\operatorname{Fl}_\Gamma(\mathcal{F}(M),Y)\subseteq\overline{\SA(M,Y)}.\]
	\end{teo}
	
	As before, we obtain a series of consequences.
	
	\begin{cor}\label{corollary:consequences_ACKrho_density}
		Let $M$ be a pointed metric space such that $\SA(M,\mathbb{R})$ is dense in $\Lip(M,\mathbb{R})$. Then, the following statements hold.
		\begin{enumerate}
		\item Let $K$ be a compact Hausdorff topological space and let $\Gamma=\{\delta_t\colon t\in K\}$ (see Remark~\ref{remark:Gamma-C(K)}). Then $\operatorname{Fl}_\Gamma(\mathcal{F}(M),C(K))\subseteq \overline{\SA(M,C(K))}$.
			\item Let $Z$ be a finite injective tensor product of Banach spaces which have $ACK_\rho$ structure. Then, if $\Gamma$ is the $1$-norming set given by Definition \ref{ACK}, we have $\operatorname{Fl}_\Gamma(\mathcal{F}(M),Z)\subseteq \overline{\SA(M,Z)}$.
			\item Let $K$ be a compact Hausdorff topological space. If $Y \in ACK_\rho$ and $\Gamma$ is the $1$-norming set given by Definition \ref{ACK}, then $\operatorname{Fl}_\Gamma(\mathcal{F}(M),C(K,Y))\subseteq \overline{\SA(M,C(K,Y))}$.
			\item Let $Y \in ACK_\rho$. If $\Gamma$ is the $1$-norming set given by Definition \ref{ACK}, then $$\operatorname{Fl}_\Gamma(\mathcal{F}(M),c_0(Y))\subseteq \overline{\SA(M,c_0(Y))},\quad \operatorname{Fl}_\Gamma(\mathcal{F}(M),\ell_\infty(Y))\subseteq \overline{\SA(M,\ell_\infty(Y))}$$ $$\text{and} \quad  \operatorname{Fl}_\Gamma(\mathcal{F}(M),c_0(Y,w))\subseteq \overline{\SA(M,c_0(Y,w))}.$$
			\item Let $K$ be a compact Hausdorff topological space. If $Y \in ACK_\rho$ and $\Gamma$ is the $1$-norming set given by Definition \ref{ACK}, then $ \operatorname{Fl}_\Gamma(\mathcal{F}(M),C_w(K,Y))\subseteq \overline{\SA(M,C_w(K,Y))}$.
		\end{enumerate}
	\end{cor}

Let us comment that, as happened in Corollary~\ref{corollary:all_consequences_ACKrho}, the following consequence also follows: if the set $\SA(M,\R)$ is dense in $\Lip(M,\R)$ and a Banach space $Y$ has property $\beta$, then $\SA(M,Y)$ is dense in $\Lip(M,Y)$. This result also appeared in \cite{ChiMar-NA}. Actually, a more general result dealing with a property weaker than property $\beta$ called property quasi-$\beta$ holds, see  \cite[Proposition~4.7]{ChiMar-NA}.

We next deal with Lipschitz compact maps. Observe that one of the disadvantages of Theorems \ref{teoACK} and \ref{ACKden} and their consequences Corollaries \ref{corollary:all_consequences_ACKrho} and \ref{corollary:consequences_ACKrho_density} is the need to deal with $\Gamma$-flat operators. But now this requirement disappears: given a Banach space $Y$, every compact operator with $Y$ as codomain is $\Gamma$-flat for every $\Gamma\subseteq B_{Y^*}$, see \cite[Example A]{ACK}. Moreover, we note that if we take a compact operator $\widehat{T}$ in the proof of Theorem \ref{teoACK}, then the operator $\widehat{S}$ that approximates $\widehat{T}$ will be also compact. Consequently, we obtain the following result.
	
	\begin{prop}\label{teoACKcompact}
		Let $M$ be a pointed metric space such that $(M,\mathbb{R})$ has the Lip-BPB property and let $Y$ be an $ACK_\rho$ Banach space. Then, the pair $(M,Y)$ has the Lip-BPB property for Lipschitz compact maps.
	\end{prop}

	Again, in view of Proposition \ref{propACK}, we obtain a series of implications.

	\begin{cor}\label{corollary:ACK-rho-compact-consequences}
		Let $M$ be a pointed metric space such that $(M,\mathbb{R})$ has the Lip-BPB property. Then, the following statements hold.
		\begin{enumerate}[(1)]
            \item Let $K$ be a compact Hausdorff topological space. Then, $(M,C(K))$ has the Lip-BPB property for Lipschitz compact maps.
			\item Let $Z$ be a finite injective tensor product of Banach spaces which have $ACK_\rho$ structure. Then, $(M,Z)$ has the Lip-BPB property for Lipschitz compact maps.
			\item Let $K$ be a compact Hausdorff topological space. If $Y \in ACK_\rho$, then $(M,C(K,Y))$ and $(M,C_w(K,Y))$ have the Lip-BPB property for Lipschitz compact maps.
			\item Let $Y \in ACK_\rho$. Then, $(M,c_0(Y))$, $(M,\ell_\infty(Y))$, and $(M,c_0(Y,w))$ have the Lip-BPB property for Lipschitz compact maps.
		\end{enumerate}
	\end{cor}
	
Let us comment that the particular case when $Y$ has property $\beta$ in the above proposition already appeared in \cite[Proposition~4.13]{ChiMar-NA}. Also, item (1) above is covered by \cite[Proposition~4.17]{ChiMar-NA}.

	As done in Proposition \ref{teoACKcompact}, the proof of Theorem \ref{ACKden} can be easily adapted to the density of Lipschitz compact maps.

	\begin{prop}\label{prop:ACK-rho-density-compact}
		Let $M$ be a pointed metric space such that $\SA(M,\mathbb{R})$ is dense in $\Lip(M,\mathbb{R})$ and $Y \in ACK_\rho$ be a Banach space. Then, $\SA_K(M,Y)$ is dense in $\Lipc(M,Y)$.
	\end{prop}

	As before, this result has many consequences.
	
	\begin{cor}\label{corollary6.9}
		Let $M$ be a pointed metric space such that $\SA(M,\mathbb{R})$ is dense in $\Lip(M,\mathbb{R})$. Then, the following statements hold.
		\begin{enumerate}
			\item Let $K$ be a compact Hausdorff topological space. Then, the set $\SA_K(M,C(K))$ is dense in $\Lipc(M,C(K))$.
			\item Let $Z$ be a finite injective tensor product of Banach spaces which have $ACK_\rho$ structure. Then, $\SA_K(M,Z)$ is dense in $\Lipc(M,Z)$.
			\item Let $K$ be a compact Hausdorff topological space. If $Y \in ACK_\rho$, then $\SA_K(M,C(K,Y))$ and $\SA_K(M,C_w(K,Y))$ are dense in $\Lipc(M,C(K,Y))$ and $\Lipc(M,C_w(K,Y))$, respectively.
			\item Let $Y \in ACK_\rho$. Then, $\SA_K(M,c_0(Y))$, $\SA_K(M,\ell_\infty(Y))$, and $\SA_K(M,c_0(Y,w))$ are dense in $\Lipc(M,c_0(Y))$, $\Lipc(M,\ell_\infty(Y))$, and $\Lipc(M,c_0(Y,w))$, respectively.
		\end{enumerate}
	\end{cor}

Let us comment another result which follows from Proposition \ref{prop:ACK-rho-density-compact}: property $\beta$ of $Y$ is enough to pass from the density of $\SA(M,\R)$ to the density of $\SA_K(M,Y)$. Actually, \cite[Proposition 4.15]{ChiMar-NA} gives a stronger result: let $M$ be a pointed metric space such that $\SA(M,\mathbb{R})$ is dense in $\Lip(M,\mathbb{R})$ and let $Y$ be a Banach space having property quasi-$\beta$ (a property weaker than property $\beta$); then $\overline{\SA_K(M,Y)}=\Lipc(M,Y)$. Besides, item (1) above is covered by \cite[Corollary 4.19]{ChiMar-NA}.
	
We would like to finish this section by presenting some more results on the Lip-BPB property for Lipschitz compact maps and on the density of strongly norm attaining Lipschitz compact maps. They will follow from a couple of results from \cite{ChiMar-NA} which were not applied there to get those consequences. We start with a result on the Lip-BPB property for Lipschitz compact maps.
	
	\begin{prop}[\textrm{\cite[Proposition 4.16]{ChiMar-NA}}]\label{projections}
		Let $M$ be a pointed metric space and let $Y$ be a Banach space. Suppose that there exists a net of norm-one projections $\{Q_\lambda\}_{\lambda \in \Lambda} \subset \operatorname{L}(Y,Y)$ such that $\{Q_\lambda(y)\}\longrightarrow y$ in norm for every $y \in Y$. If there is a function $\eta\colon \mathbb{R}^+ \longrightarrow \mathbb{R}^+$ such that for every $\lambda \in \Lambda$, the pair $(M,Q_\lambda(Y))$ has the Lip-BPB property for Lipschitz compact maps witnessed by the function $\eta$, then the pair $(M,Y)$ has the Lip-BPB property for Lipschitz compact maps.
	\end{prop}

	The following result collects several consequences of the proposition above. None of them was included in \cite{ChiMar-NA}. Observe that item (1) below extends items (1) and (3) of our Corollary \ref{corollary:ACK-rho-compact-consequences}.
	
	\begin{cor}\label{codomaincompact}
		Let $M$ be a pointed metric space and let $Y$ be a Banach space such that the pair $(M,Y)$ has the Lip-BPB property for Lipschitz compact maps.
		\begin{enumerate}
			\item  For every compact Hausdorff topological space $K$, the pair $(M,C(K,Y))$ has the Lip-BPB property for Lipschitz compact maps.
			\item For $1\leq p<\infty$, if the pair $(M,\ell_p(Y))$ has the Lip-BPB property for Lipschitz compact maps, then so does $(M,L_p(\mu,Y))$ for every positive measure $\mu$.
			\item For every $\sigma$-finite positive measure $\mu$, the pair $(M,L_\infty(\mu,Y))$ has the Lip-BPB property for Lipschitz compact maps.
		\end{enumerate}
	\end{cor}

	\begin{proof}
		This proof is based on the proof of Theorem 3.15 in \cite{dgmm}. To prove (1),  following the proof of Theorem 4 in \cite{Johnson}, by using peak partitions of the unit we can find a net $\{Q_\lambda\}_{\lambda \in \Lambda}$ of norm-one projections on $C(K,Y)$ such that $\{Q_\lambda(f)\}\longrightarrow f$ in norm for every $f \in C(K,Y)$ and $Q_\lambda(C(K,Y))$ is isometrically isomorphic to $\ell_\infty^m(Y) $. Consequently, (1) follows from Propositions \ref{prop6.15} and \ref{projections}.
		
		In order to prove (2), fix $1\leq p<\infty$. If $L_1 (\mu)$ is finite-dimensional, the result is a consequence of Proposition \ref{recdomain}. Otherwise, by using Lemma 3.12 in \cite{dgmm} we may find a net $\{Q_\lambda\}_{\lambda \in \Lambda}$ of norm-one projections on $L_p(\mu,Y)$ such that $\{Q_\lambda\}\longrightarrow f$ in norm for every $f \in L_p(\mu,Y)$ and $Q_\lambda(L_p(\mu,Y))$ is isometrically isomorphic to $\ell_p(Y)$. Therefore, it is enough to apply Proposition \ref{projections}.
		
		As before, if $L_\infty (\mu)$ is finite-dimensional, the result is a consequence of Proposition \ref{prop6.15}. Otherwise, if $L_\infty(\mu)$ is infinite-dimensional, we may suppose that the measure is finite by using Proposition 1.6.1 in \cite{cm}. Then, Lemma 3.12 in \cite{dgmm} provides a net $\{Q_\lambda\}_{\lambda \in \Lambda}$ of norm-one projections on $L_\infty(\mu,Y)$ such that $\{Q_\lambda\}\longrightarrow f$ in norm for every $f \in L_\infty(\mu,Y)$ and $Q_\lambda(L_p(\mu,Y))$ is isometrically isomorphic to $\ell_\infty(Y)$. Consequently, the result follows from Propositions \ref{prop6.15} and \ref{projections}.\qedhere
	\end{proof}

An analogous result to Proposition~\ref{projections} also appeared in \cite{ChiMar-NA} for the density of strongly norm attaining Lipschitz compact maps.
	
	\begin{prop}[\textrm{\cite[Proposition~4.18]{ChiMar-NA}}]\label{prop_Qlambda_density}
		Let $M$ be a pointed metric space and $Y$ be a Banach space. Suppose that there exists a net of norm-one projections $\{Q_\lambda\}_{\lambda \in \Lambda} \subset \operatorname{L}(Y,Y)$ such that $\{Q_\lambda(y)\}\longrightarrow y$ in norm for every $y \in Y$. If $\SA_K(M,Q_\lambda(Y))$ is dense in $\Lipc(M,Q_\lambda(Y))$ for every $\lambda\in \Lambda$, then
		\[\overline{\SA_K(M,Y)}=\Lipc(M,Y).\]
	\end{prop}

Now, by using this proposition instead of Proposition \ref{projections} and replacing the necessary results by their analogous versions with respect to $\SA_K(M,Y)$, the proof of  Corollary \ref{codomaincompact} can be easily adapted to get the following results about the density of strongly norm attaining Lipschitz compact maps.

None of the results appeared in \cite{ChiMar-NA}. Besides, item (1) below extends items (1) and (3) of our Corollary~\ref{corollary6.9}.
	
	\begin{cor}\label{codomainSNAcompact}
		Let $M$ be a pointed metric space and let $Y$ be a Banach space such that $\SA_K(M,Y)$ is dense in $\Lipc(M,Y)$.
		\begin{enumerate}
			\item $\SA_K(M,C(K,Y))$ is dense in $\Lipc(M,C(K,Y))$ for every compact Hausdorff topological space $K$.
			\item For $1\leq p<\infty$, if $\SA_K(M,\ell_p(Y))$ is dense in $\Lipc(M,\ell_p(Y))$, then we have that the set $\SA_K(M,L_p(\mu,Y))$ is dense in $\Lipc(M,L_p(\mu,Y))$ for every positive measure $\mu$.
			\item $\SA_K(M,L_\infty(\mu,Y))$ is dense in $\Lipc(M,L_\infty(\mu,Y))$ for every $\sigma$-finite positive measure $\mu$.
		\end{enumerate}
	\end{cor}

\begin{proof}
	We proceed as in the proof of Corollary \ref{codomaincompact}. To prove (1), Theorem 4 in \cite{Johnson} shows that we can find a net $\{Q_\lambda\}_{\lambda \in \Lambda}$ of norm-one projections on $C(K,Y)$ such that $\{Q_\lambda(f)\} \longrightarrow f$ in norm for every $f \in C(K,Y)$ and $Q_\lambda(C(K,Y))$ is isometrically isomorphic to $\ell_p(Y)$. Consequently, we can apply Propositions \ref{prop6.16} and \ref{prop_Qlambda_density} to obtain the result.
	
	In order to prove (2), fix $1\leq p<\infty$. If $\operatorname{L}_1(\mu)$ is finite-dimensional, the result is a consequence of Proposition \ref{absteo}. Otherwise, using Lemma 3.12 in \cite{dgmm} we find a net $\{Q_\lambda\}_{\lambda \in \Lambda}$ of norm-one projections on $L_p(\mu,Y)$ such that $\{Q_\lambda\}\longrightarrow f$ in norm for every $f \in L_p(\mu,Y)$ and $Q_\lambda(L_p(\mu,Y))$ is isometrically isomorphic to $\ell_p(Y)$. Consequently, we can apply Proposition \ref{prop_Qlambda_density}.
	
	Finally, if $L_\infty (\mu)$ is finite-dimensional, the result follows from Proposition \ref{prop6.16}. If $L_\infty(\mu)$ is infinite-dimensional, we may suppose that the measure is finite by using Proposition 1.6.1 in \cite{cm}. Then, Lemma 3.12 in \cite{dgmm} provides a net $\{Q_\lambda\}_{\lambda \in \Lambda}$ of norm-one projections on $L_\infty(\mu,Y)$ such that $\{Q_\lambda\}\longrightarrow f$ in norm for every $f \in L_\infty(\mu,Y)$ and $Q_\lambda(L_p(\mu,Y))$ is isometrically isomorphic to $\ell_\infty(Y)$. Consequently, we can apply Propositions \ref{prop6.16} and \ref{prop_Qlambda_density} to get the result.\qedhere
\end{proof}

\section{Absolute sums of codomains}
In this last section we study the behavior of the Lip-BPB property and the density of $\SA(M,Y)$ with respect to absolute sums of the codomain. We need some definitions.
		
	\begin{defi}\label{abssum}
		An \textit{absolute norm} is a norm $|\cdot|_a$ in $\mathbb{R}^2$ such that $$|(1,0)|_a=|(0,1)|_a=1 \quad \text{ and }\quad |(s,t)|_a=|(|s|,|t|)|_a \text{ for every $s,t \in \mathbb{R}$.}
$$
		Given two Banach spaces $W$ and $Z$ and an absolute norm $|\cdot|_a$, the \textit{absolute sum} of $W$ and $Z$ with respect to $|\cdot|_a$, denoted by $W \oplus_a Z$, is the Banach space $W\times Z$ endowed with the norm
		\[ \|(w,z)\|_a=|(\|w\|,\|z\|)|_a\quad \forall \, w \in W,\quad \forall \, z \in Z. \]
A closed subspace $Y_1$ of a Banach space $Y$ is said to be an \emph{absolute summand} of $Y$ whenever there exists a closed subspace $Z$ of $Y$ and an absolute norm $|\cdot|_a$ in $\mathbb{R}^2$ such that $Y\equiv Y_1\oplus_a Z$.
	\end{defi}	
		
	 We will need the next easy lemma (for a proof, see Lemma 2.2 in \cite{garcia-pacheco}, for instance).

	\begin{lem}
		Let $W$ and $Z$ be Banach spaces and $|\cdot|_a$ be any absolute norm in $\mathbb{R}^2$. If $(w,z) \in S_{W \oplus_a Z}$ and $(w^*,z^*) \in S_{W^* \oplus_{a^*} Z^*}$ are such that $\langle (w,z),(w^*,z^*)\rangle=1$, then
		\[w^*(w)=\|w^*\|\|w\|\quad \mbox{ and } \quad z^*(z)=\|z^*\|\|z\|. \]
	\end{lem}

Our first result is following lifting of the Lip-BPB property from a space to its absolute summands.

	\begin{prop}\label{recdomain}
		Let $M$ be a pointed metric space, $Y$ be a Banach space and $Y_1$ be an absolute summand of $Y$. If the pair $(M,Y)$ has the Lip-BPB property with a function $\eps\longmapsto \eta(\eps)$, then so does $(M,Y_1)$ with the same function.
	\end{prop}
	
Let us comment that the case of $\oplus_a = \oplus_\infty$ already appeared in \cite[Lemma~4.8]{ChiMar-NA}.

	\begin{proof}
		Fix $\varepsilon >0$ and consider $\widehat{F}_1 \in \operatorname{L}(\mathcal{F}(M),Y_1)$ with $\|F_1\|_L=1$ and $m\in\Mol(M)\subset S_{\F(M)}$ satisfying that
$$\|\widehat{F}_1(x_0)\|>1-\eta(\varepsilon).
$$
Let us define the operator
$\widetilde{T}\in \operatorname{L}(\mathcal{F}(M),Y_1)$ by $\widetilde{T}(x)=(\widehat{F}_1(x),0)$ for all $x\in \mathcal{F}(M)$, and note that it satisfies that $\|\widetilde{T}\|=1$ and
$$
\|\widetilde{T}(m)\|=\|(\widehat{F}_1(m),0)\|_a=\|\widehat{F}_1(m)\|>1-\eta(\eps).
$$
Now, following the proof of Theorem 2.1 in \cite{absolutesums} by using the Lip-BPB property of $(M,Y)$ instead of the BPBp of $(\mathcal{F}(M),Y)$, we may
construct $\widehat{G}_1 \in \operatorname{L}(\mathcal{F}(M),Y_1)$ and $m' \in \Mol(M)$ such that
		\[ \|\widehat{G}_1(m')\|=\|G_1\|_L=1,\qquad \|F_1-G_1\|_L< \varepsilon, \qquad \|m-m'\|< \varepsilon.\]
But this implies that $(M,Y_1)$ has the Lip-BPB property, as desired.
	\end{proof}
	
	Note that, as it is proved in the above proposition, the same function $\eta$ from the Lip-BPB property of $(M,Y)$ works for the Lip-BPB property of $(M,Y_1)$. This is the key fact to obtain the following consequence.
	
\begin{cor}
		Let $M$ be a pointed metric space such that $(M,Y)$ has the Lip-BPB property for all Banach spaces $Y$. Then, there exists a function $\eta_M(\varepsilon)$, which only depends on $M$, such that the pair $(M,Y)$ has the Lip-BPB property witnessed by the function $\eta_M(\varepsilon)$ for every Banach space $Y$.
	\end{cor}

\begin{proof}
Suppose this is not the case. Then there is a sequence $Y_n$ of Banach spaces such that whenever each pair $(M,Y_n)$ has the Lip-BPB property witnessed by a function $\eta_n(\varepsilon)>0$, one has that $\inf_n \eta_n(\varepsilon)=0$ for every $0<\varepsilon<1$. Then, consider the space $Y=\left[\bigoplus_{n \in \N} Y_n\right]_{c_0}$ and observe that, by hypothesis, the pair $(M,Y)$ has the Lip-BPB property witnessed by a function $\varepsilon\longmapsto \eta(\varepsilon)>0$. As each $Y_n$ is clearly an absolute summand of $Y$, it follows by Proposition \ref{recdomain} that for every $n\in \N$, each pair $(M,Y_n)$ has the Lip-BPB property witnessed by the function $\varepsilon\longmapsto \eta(\varepsilon)>0$, a contradiction to our assumption.
\end{proof}

For the density of $\SA(M,Y)$, we can give the following result whose proof is a modification of \cite[Proposition 2.5]{absolutesums}, in the same way as the proof of Proposition \ref{recdomain} follows from \cite[Theorem 2.1]{absolutesums}.
	
\begin{prop}\label{absteo}
		Let $M$ be a pointed metric space, let $Y$ be a Banach space, and let $Y_1$ be an absolute summand of $Y$. If $\SA(M,Y)$ is dense in $\Lip(M,Y)$, then $\SA(M,Y_1)$ is dense in $\Lip(M,Y_1)$.
	\end{prop}
	
Another result in the same direction is the following modification of Proposition 2.8 in \cite{acklm}.

\begin{prop}\label{C(K)}
	Let $M$ be a pointed metric space, $Y$ be a Banach space, and $K$ be a compact Hausdorff topological space. If $(M,C(K,Y))$ has the Lip-BPB property witnessed by a function $\eta(\varepsilon)$, then $(M,Y)$ has the Lip-BPB property witnessed by the same function.
\end{prop}

\begin{proof}
	Fix $\varepsilon>0$ and take $\eta(\varepsilon)$ the constant from the Lip-BPB property of $(M,C(K,Y))$.  Consider $\widehat{F}_1 \in \operatorname{L}(\mathcal{F}(M),Y)$ with $\|\widehat{F}_1\|=1$ and $m \in \Mol(M)$ satisfying
	\[ \|\widehat{F}_1(m)\|>1-\eta(\varepsilon).\]
	Let us define $\widehat{F}\colon \mathcal{F}(M) \longrightarrow C(K,Y)$ given by $$[\widehat{F}(x)](t)=\widehat{F}_1(x) \quad \text{ for every $x \in \mathcal{F}(M)$, $t \in K$.}
$$
Then, it is clear that $\|\widehat{F}\|=\|\widehat{F}_1\|=1$. Furthermore, $\|\widehat{F}(m)\|>1-\eta(\varepsilon)$. By the assumption, there exist $\widehat{G} \in \operatorname{L}(\mathcal{F}(M),C(K,Y)) $ and $u \in \Mol(M)$ such that
	\[ \|\widehat{G}(u)\|=\|\widehat{G}\|=1, \quad \|\widehat{F}-\widehat{G}\|<\varepsilon, \quad \|m-u\|<\varepsilon.\]
	Moreover, since $K$ is compact, there is $t_1\in K$ such that $1=\|\widehat{G}(u)\|=\|[\widehat{G}(u)](t_1)\|$. Now, let us define $\widehat{G}_1\colon \F(M) \longrightarrow Y$ by $\widehat{G}_1(x)=[\widehat{G}(x)](t_1)$ for every $x \in \mathcal{F}(M)$. Note that
	\[ \|\widehat{G}_1\|=\sup_{x\in B_{\mathcal{F}(M)}} \bigl\|[\widehat{G}(x)](t_1)\bigr\| = \bigl\|[\widehat{G}(u)](t_1)\bigr\| = \bigl\|\widehat{G}_1(u)\bigr\|=1.\]
	In addition, we have that
	\begin{align*}
	\|G_1-F_1\|_L&=\sup_{x\in B_{\mathcal{F}(M)}} \left\{\bigl\|[\widehat{G}(x)](t_1)-[\widehat{F}(x)](t_1)\bigr\|\right\}\\
	&\leq \sup_{x \in B_{\mathcal{F}(M)}}\left\{\bigl\|\widehat{G}(x)-\widehat{F}(x)\bigr\|\right\}= \bigl\|\widehat{G}-\widehat{F}\bigr\|<\varepsilon.
	\end{align*}
	As we already know that $\|m-u\|<\varepsilon$, we obtain that $(M,Y)$ has the Lip-BPB property witnessed by the function $\eta(\varepsilon)$.
\end{proof}

The previous proposition also has an analogous formulation for the density of strongly norm-attaining Lipschitz maps.

	\begin{prop}
		Let $M$ be a pointed metric space, let $Y$ be a Banach space, and let $K$ be a compact Hausdorff topological space. Assume that $\SA(M,C(K,Y))$ is dense in $\Lip(M,C(K,Y))$. Then, $\SA(M,Y)$ is dense in $\Lip(M,Y)$.
	\end{prop}

	\begin{proof}
		Given $\varepsilon>0$, consider $\widehat{F}_1 \in \operatorname{L}(\mathcal{F}(M),Y)$ with $\|F_1\|_L=1$. Let us define $\widehat{F}$ as in the proof of Proposition \ref{C(K)}. By hypothesis, there exist $\widehat{G} \in \operatorname{L}(\mathcal{F}(M),C(K,Y))$ and $u \in \Mol(M)$ such that
		\[ \|\widehat{G}(u)\|=\|\widehat{G}\|=1 \mbox{ and } \quad \|\widehat{G}-\widehat{F}\|<\varepsilon.\]
		Since $K$ is compact, there is $t_1 \in K$ such that $1 =\|\widehat{G}(u)\|=\|[\widehat{G}(u)](t_1)\|$. Now, let us define the linear and bounded operator $\widehat{G}_1 \colon \mathcal{F}(M) \longrightarrow Y$ given by $\widehat{G}_1(x)=[\widehat{G}(x)](t_1)$ for every $x \in \mathcal{F}(M)$. By repeating the argument in Proposition \ref{C(K)}, we obtain that $\widehat{G}_1$ attains its norm at $u \in\Mol(M)$ and $\|\widehat{G}_1-\widehat{F}_1\|\leq \|\widehat{G}-\widehat{F}\|<\varepsilon$. Consequently, we obtain that $\SA(M,Y)$ is dense in $\Lip(M,Y)$.
	\end{proof}

A cautious inspection of the proofs in this section shows that if one starts with a Lipschitz compact map, then one also gets a Lipschitz compact map in each case. Then, every result has its own version for the Lip-BPB property for Lipschitz compact maps and for the density of strongly norm attaining Lipschitz compact maps. We summarize all the results in the following proposition.

\begin{prop}
Let $M$ be a pointed metric space and let $Y$ be a Banach space.
\begin{enumerate}[(a)]
\item Let $Y_1$ be an absolute summand of $Y$. If the pair $(M,Y)$ has the Lip-BPB property for Lipschitz compact maps with a function $\eps\longmapsto \eta(\eps)$, then so does $(M,Y_1)$.
\item If for all Banach spaces $Z$ the pair $(M,Z)$ has the Lip-BPB property for Lipschitz compact maps, then there exists a function $\eta(\varepsilon)$, which only depends on $M$, such that for every Banach space $Z$ the pair $(M,Z)$ has the Lip-BPB property for Lipschitz compact maps witnessed by the function $\eta(\varepsilon)$.
\item Let $Y_1$ be an absolute summand of $Y$. If the set $\SA_K(M,Y)$ is dense in $\Lipc(M,Y)$, then $\SA_K(M,Y_1)$ is dense in $\Lipc(M,Y_1)$.
\item If for some compact Hausdorff space $K$ the pair $(M,C(K,Y))$ has the Lip-BPB property for Lipschitz compact maps witnessed by a function $\eta(\varepsilon)$, then $(M,Y)$ has the Lip-BPB property for Lipschitz compact maps witnessed by the same function.
\item If $\SA_K(M,C(K,Y))$ is dense in $\Lipc(M,C(K,Y))$ for some compact Hausdorff space $K$, then $\SA_K(M,Y)$ is dense in $\Lipc(M,Y)$.
\end{enumerate}
\end{prop}

\begin{proof}
	\begin{enumerate}[(a)]
		\item Looking at the proof of Proposition \ref{recdomain}, we see that if the operator $\widehat{F}_1$ is compact, then the operator $\widehat{T} \in \operatorname{L}(\mathcal{F}(M),Y_1)$ given by $\widehat{T}(x)=(\widehat{F}_1(x),0)$ for all $x \in \mathcal{F}(M)$ is also compact. Then, we just have to follow that proof using that $(M,Y)$ has the Lip-BPB property for Lipschitz compact maps.
		\item This is a direct consequence of the fact that the same function $\eta$ from the Lip-BPB property for Lipschitz compact maps of $(M,Y)$ works for the Lip-BPB property for Lipschitz compact maps of $(M,Y_1)$ for every summand $Y_1$ of $Y$.
		\item We just have to proceed as in part (a), but forgetting about the requirement on the molecules to satisfy $\|m-m'\|<\varepsilon$.
		\item Observe that if the operator $\widehat{F}_1$ considered in the proof of Proposition \ref{C(K)} is compact, then so is the operator $\widehat{F}$. Following that proof we see that if we apply our assumption, then the operator $\widehat{G}$ is compact. Consequently, the operator $\widehat{G}_1$, which satisfies the conditions we wanted, is also compact.
		\item This is a slight modification of the previous item. We just have to forget about the requirement on the molecules to satisfy $\|m-u\|<\varepsilon$.\qedhere
	\end{enumerate}
\end{proof}

Let $M$ be a pointed metric space, let $\{Y_i\}_{i\in I}$ be a family of Banach spaces and let $Y=[\bigoplus_{i\in I} Y_i]_{c_0}$ or $Y=[\bigoplus_{i\in I} Y_i]_{\ell_\infty}$. By Proposition~\ref{recdomain}, if the pair $(M,Y)$ has the Lip-BPB property, then all the pairs $(M,Y_i)$ have the Lip-BPB property witnessed by the same function. By Proposition~\ref{absteo}, if $\SA(M,Y)$ is dense in $\Lip(M,Y)$, then $\SA(M,Y_i)$ is dense in $\Lip(M,Y_i)$ for all $i\in I$. Our next aim is to show that the reversed results follow. We start with the Lip-BPB property.
	
		\begin{prop}\label{Lip-BPBp domain}
		Let $M$ be a pointed metric space, let $\{Y_i\}_{i\in I}$ be a family of Banach spaces, and let $Y$ be $Y=[\bigoplus_{i\in I} Y_i]_{c_0}$ or $Y=[\bigoplus_{i\in I} Y_i]_{\ell_\infty}$. Assume that $(M,Y_i)$ has the Lip-BPB property witnessed by a function $\eta_i(\varepsilon)$ for every $i \in I$. If $\inf\{\eta_i(\varepsilon)\colon i \in I\}>0$ for every $\varepsilon>0$, then $(M,Y)$ has the Lip-BPB property.
	\end{prop}
	
	\begin{proof}
Fix $\varepsilon>0$, take $\eta(\varepsilon):=\inf\{\eta_i(\varepsilon)\colon i \in I\}>0$ and note that we have $\eta_i(\varepsilon)\geq \eta(\varepsilon)$ for every $i \in I$. Consider $Q_i \colon Y \longrightarrow Y_i$ the natural projection and $E_i\colon Y_i \longrightarrow Y$ the natural embedding for every $i \in I$. Take $\widehat{F} \in \operatorname{L}(\mathcal{F}(M),Y)$ with $\|F\|_L=1$ and $m \in \Mol(M)$ such that
		\[ \|\widehat{F}(m)\|>1-\eta(\varepsilon).\]
		Then, there exists $k \in I$ so that $\|Q_k\widehat{F}(m)\|>1-\eta(\varepsilon)$. By hypothesis, there exist $\widehat{G}_k \in \operatorname{L}(\mathcal{F}(M),Y_k)$ and $u \in \Mol(M)$ satisfying
		\[ \|\widehat{G}_k(u)\|=\|G_k\|_L=1,\quad \|Q_k\widehat{F}-\widehat{G}_k\|<\varepsilon, \quad \|m-u\|<\varepsilon.\]
		Now, let us define $\widehat{G} \colon \mathcal{F}(M) \longrightarrow Y$ given by
		\[ \widehat{G}(x)=\sum_{i\neq k} E_i(Q_i(\widehat{F}))(x) + E_k \widehat{G}_k(x) \quad \forall \, x \in \F(M). \]
		Then, we have that $\|G\|_L\leq 1$ and $\|\widehat{G}(u)\|\geq \|\widehat{G}_k(u)\|=1$. Therefore, $\widehat{G}$ attains its norm at $u \in \Mol(M)$. Finally, note that
		\[ \|F-G\|_L=\sup\{\|Q_i(\widehat{F}-\widehat{G})\|\colon i \in I\}=\|Q_k(\widehat{F}-\widehat{G})\|<\varepsilon,\]
		that is, $(M,Y)$ has the Lip-BPB property.
	\end{proof}

 	It is possible to give a result analogous to Proposition \ref{Lip-BPBp domain} for the density of $\SA(M,Y)$.
	
	\begin{prop}\label{c_0den}
Let $M$ be a pointed metric space, let $\{Y_i\}_{i\in I}$ be a family of Banach spaces, and let $Y$ be $Y=[\bigoplus_{i\in I} Y_i]_{c_0}$ or $Y=[\bigoplus_{i\in I} Y_i]_{\ell_\infty}$. If $\overline{\SA(M,Y_i)}=\Lip(M,Y_i)$ for every $i \in I$, then $$\overline{\SA(M,Y)}=\Lip(M,Y).$$
	\end{prop}

	\begin{proof}
		For each $i \in I$, consider $Q_i \colon Y \longrightarrow Y_i$ the natural projection and $E_i\colon Y_i \longrightarrow Y$ the natural embedding. Fix $\varepsilon>0$ and $\widehat{F} \in \operatorname{L}(\mathcal{F}(M),Y)$ with $\|F\|_L=1$. There exists $k \in I$ so that $\|Q_k \widehat{F}\|>1-\frac{\varepsilon}{2}$. Then, since $\overline{\SA(M,Y_k)}=\Lip(M,Y_K)$ we may find $G_k\in \Lip(M,Y_k)$ and $u \in \Mol(M)$ such that
		\[ \|\widehat{G}_k(u)\|=\|G_k\|_L=1,\quad \|\widehat{G}_k-Q_k\widehat{F}\|<\varepsilon.\]
			Now, let us define $\widehat{G} \colon \mathcal{F}(M) \longrightarrow Y$ given by
		\[ \widehat{G}(x)=\sum_{i\neq k} E_i(Q_i(\widehat{F}))(x) + E_k \widehat{G}_k(x) \quad \forall \, x \in \F(M). \]
		Then, we have that $\|G\|_L\leq 1$ and $\|\widehat{G}(u)\|\geq \|\widehat{G}_k(u)\|=1$. Therefore, $\widehat{G}$ attains its norm at $u$. Finally, note that
		\[ \|F-G\|_L=\sup\{\|Q_i(\widehat{F}-\widehat{G})\|\colon i \in I\}=\|Q_k(\widehat{F}-\widehat{G})\|<\varepsilon.\qedhere\]
	\end{proof}

A sight to the above two proofs shows that the analogous results for Lipschitz compact maps are also valid.

	\begin{prop}\label{prop6.15}
Let $M$ be a pointed metric space, let $\{Y_i\}_{i\in I}$ be a family of Banach spaces, and let $Y$ be $Y=[\bigoplus_{i\in I} Y_i]_{c_0}$ or $Y=[\bigoplus_{i\in I} Y_i]_{\ell_\infty}$. Assume that for each $i \in I$ the pair $(M,Y_i)$ has the Lip-BPB property for Lipschitz compact maps witnessed by a function $\eta_i(\varepsilon)$. If $\inf\{\eta_i(\varepsilon)\colon i \in I\}>0$ for every $\varepsilon>0$, then $(M,Y)$ has the Lip-BPB property for Lipschitz compact maps.
	\end{prop}

	\begin{prop}\label{prop6.16}
Let $M$ be a pointed metric space, let $\{Y_i\}_{i\in I}$ be a family of Banach spaces, and let $Y$ be $Y=[\bigoplus_{i\in I} Y_i]_{c_0}$ or $Y=[\bigoplus_{i\in I} Y_i]_{\ell_\infty}$. If $\SA_K(M,Y_i)$ is dense in $\Lipc(M,Y_i)$ for every $i \in I$, then $\SA_K(M,Y)$ is dense in $\Lipc(M,Y)$.
	\end{prop}

\vspace*{0.5cm}

\noindent \textbf{Acknowledgments:\ } The authors would like to thank A.~Avil\'{e}s, L.~C.~Garc\'{\i}a-Lirola, V.~Kadets, P.~Koszmider, and A.~Rueda Zoca for kindly answering some inquiries about the contents of the paper. They also thanks G.~Choi, Y.~S.~Choi, and M.~Jung for some comments on the content of the paper.

\end{document}